\documentclass[11pt,a4paper]{article}
\usepackage{fullpage}

\usepackage{lmodern}
\usepackage{slantsc}
\usepackage[dvipsnames]{xcolor}
\usepackage[utf8]{inputenc}
\usepackage{amsthm,amssymb}
\usepackage[leqno]{amsmath}
\usepackage{mathtools}
\usepackage{thmtools}
\usepackage{subcaption,thm-restate}
\usepackage{caption}

\usepackage[mathcal,mathscr]{eucal}
\usepackage{tcolorbox}
\usepackage{epsfig,graphicx,graphics,color}
\usepackage{enumerate}
\usepackage{enumitem, crossreftools}
\usepackage[sort]{cite}
\usepackage{titling}
\usepackage{hyperref}
\usepackage{tikz}
\hypersetup{
    colorlinks=true,       
    linkcolor=blue,        
    citecolor=red,         
    filecolor=magenta,     
    urlcolor=cyan,         
    linktocpage=true
}

\theoremstyle{plain}
\newtheorem{thm}{Theorem}

\newtheorem{cl}[thm]{Claim}

\theoremstyle{definition}

\def\final{0}  
\def\iflong{\iffalse}
\ifnum\final=0  
\newcommand{\knote}[1]{{\color{red}[{\tiny \textbf{Krist\'of:} \bf #1}]\marginpar{\color{red}*}}}
\newcommand{\gnote}[1]{{\color{blue}[{\tiny \textbf{Gergő:} \bf #1}]\marginpar{\color{blue}*}}}
\newcommand{\tnote}[1]{{\color{green}[{\tiny \textbf{Tam\'as:} \bf #1}]\marginpar{\color{green}*}}}
\else 
\newcommand{\knote}[1]{}
\newcommand{\gnote}[1]{}
\newcommand{\tnote}[1]{}
\fi

\newcommand{\cI}{\mathcal{I}}

\newcommand{\cJ}{\mathcal{J}}

\newtcolorbox{probbox}{arc=6pt,
                      colback=white!100,
                      colframe=black!50,
                      before skip=6pt,
                      after skip=6pt,
                      boxsep=1pt,
                      left=6pt,
                      right=6pt,
                      top=4pt,
                      bottom=4pt}

\newcommand{\searchprob}[3]{
   \begin{center}%
    \begin{minipage}{\linewidth}%
      \begin{probbox}
      \textsc{#1}\\[0.2ex]
      \textbf{Input:} #2\\[0.2ex]
      \textbf{Goal:} #3
      \end{probbox}
    \end{minipage}%
  \end{center}
}

\usepackage{authblk}

\date{}

\title{On the complexity of packing rainbow spanning trees}

\thanksmarkseries{alph}
\author{Kristóf Bérczi}
\author{Gergely Csáji}
\author{Tamás Király}
\affil{{\footnotesize MTA-ELTE Momentum Matroid Optimization Research Group and MTA-ELTE Egerv\'ary Research Group, Department of Operations Research, E\"otv\"os Lor\'and University, Budapest, Hungary.\\ Email: \texttt{kristof.berczi@ttk.elte.hu, csajigergely@student.elte.hu, tamas.kiraly@ttk.elte.hu}.}}

\begin{document}
\maketitle

\begin{abstract}
One of the most important questions in matroid optimization is to find disjoint common bases of two matroids. The significance of the problem is well-illustrated by the long list of conjectures that can be formulated as special cases. B\'erczi and Schwarcz showed that the problem is hard in general, therefore identifying the borderline between tractable and intractable instances is of interest. 

In the present paper, we study the special case when one of the matroids is a partition matroid while the other one is a graphic matroid. This setting is equivalent to the problem of packing rainbow spanning trees, an extension of the problem of packing arborescences in directed graphs which was answered by Edmonds' seminal result on disjoint arborescences. We complement his result by showing that it is NP-complete to decide whether an edge-colored graph contains two disjoint rainbow spanning trees. Our complexity result holds even for the very special case when the graph is the union of two spanning trees and each color class contains exactly two edges. As a corollary, we give a negative answer to a question on the decomposition of oriented $k$-partition-connected digraphs.

\medskip

\end{abstract}

\section{Introduction}
\label{sec:intro}

The \emph{covering number} of a matroid $M$, denoted by $\beta(M)$, is the minimum number of independent sets needed to cover its ground set. The matroid partition theorem of Edmonds and Fulkerson~\cite{edmonds1965transversals} implies that $\beta(M)=\max\{\lceil |X|/r(X)\rceil:\ \emptyset\neq X\subseteq S\}$. Analogously, given two matroids $M_1=(S,r_1)$ and $M_2=(S,r_2)$, the \emph{covering number $\beta(M_1,M_2)$ of their intersection} is the minimum number of common independent sets needed to cover $S$. It is not difficult to see that $\beta(M_1,M_2)\leq\beta(M_1)\cdot\beta(M_2)$ holds, but this gives a very weak upper bound on the covering number of the intersection. Thus, a central problem of matroid theory is to find upper bounds on $\beta(M_1,M_2)$ in terms of $\beta(M_1)$ and $\beta(M_2)$. 

By replacing one of the matroids by a general simplicial complex, Aharoni and Berger~\cite{aharoni2006intersection} provided a generalization of Edmonds' matroid intersection theorem. As an application, they showed that $\beta(M_1,M_2)\leq 2\max\{\beta(M_1),\allowbreak \beta(M_2)\}$, that is, the covering number of the intersection is at most twice the maximum of the covering numbers of the matroids. Establishing a constant multiplicative gap between the lower and the upper bounds on the covering number of the intersection of matroids was a significant milestone. Nevertheless, no example is known for which the true value would be even close to the upper bound. In fact, Aharoni and Berger conjectured that \emph{if $M_1$ and $M_2$ are loopless matroids on the same ground set, then $\beta(M_1,M_2)\leq\max\{\beta(M_1),\beta(M_2)\}+1$}. The conjecture was shown to be true when $\max\{\beta(M_1),\beta(M_2)\}\leq 2$~\cite{aharoni2012edge}, and no counterexample is known even for the stronger statement $\beta(M_1,M_2)\leq\max\{\beta(M_1),\beta(M_2)+1\}$. 

A packing counterpart of the problem of covering by independent sets is to find disjoint common bases of two matroids. The \emph{packing number} of a matroid $M$, denoted by $\gamma(M)$, is the maximum number of its pairwise disjoint bases. The matroid partition theorem of Edmonds and Fulkerson~\cite{edmonds1965transversals} implies $\gamma(M)=\min\{\lfloor |S-X|/(r(S)-r(X))\rfloor:\ X\subseteq S,r(X)<r(S)\}$. Given two matroids $M_1=(S,r_1)$ and $M_2=(S,r_2)$, the \emph{packing number $\gamma(M_1,M_2)$ of their intersection} is the maximum number of pairwise disjoint common bases. An easy upper bound for the packing number of the intersection is $\gamma(M_1,M_2)\leq\min\{\gamma(M_1),\gamma(M_2)\}$. However, unlike in the case of the covering number, no good bounds are known from the opposite direction, and giving lower bounds for $\gamma(M_1,M_2)$ in terms of $\gamma(M_1)$ and $\gamma(M_2)$ is an intriguing open problem. To be more precise, it might easily happen that the two matroids have no common bases at all; an easy example is when each element is a loop in at least one of them. Therefore, it is common to concentrate on instances where the ground set partitions into bases in both matroids.

It is worth mentioning that, by a result of Harvey, Kir\'aly, and Lau~\cite{harvey2011disjoint}, one of the matroids can be assumed to be a partition matroid, and Edmonds' matroid intersection theorem~\cite{edmonds1970submodular} implies that $\gamma(M_1,M_2)\geq 1$. Showing the existence of two common bases is already challenging, and it requires non-trivial ideas even in special cases such as proper edge-colorings of complete graphs \cite{brualdi1996multicolored} or Rota's conjecture \cite{geelen2007rota}. An analogous finding for three common bases would be a landmark result towards answering Woodall's conjecture \cite{woodall1978menger}. In \cite{aharoni2006intersection}, Aharoni and Berger considered a relaxation of the problem in which disjoint common generators are required instead of bases, and showed that there always exist $\lfloor \min\{\gamma(M_1),\gamma(M_2)\}/2 \rfloor$ pairwise disjoint common generators of $M_1$ and $M_2$.

\paragraph{Previous work.}

B\'erczi and Schwarcz~\cite{berczi2021complexity} proved that there is no algorithm which decides if the ground set of two matroids can be partitioned into common bases by using a polynomial number of independence queries. Their result implies that determining the exact values of $\beta(M_1,M_2)$ and $\gamma(M_1,M_2)$ for two matroids is hard under the rank oracle model. Nevertheless, the hardness of the abstract problem has no implications on the complexity of its special cases.

A particularly well-investigated special case of packing common bases is the intersection of the graphic matroid of a complete graph $K_n$ on $n$ vertices and a partition matroid. By thinking of the partition classes as color classes, this problem can also be interpreted as finding disjoint rainbow spanning trees of an edge-colored complete graph. Here, a spanning tree is called rainbow colored if no two of its elements have the same color. Brualdi and Hollingsworth \cite{brualdi1996multicolored} conjectured that if $k\geq 3$ and each color class forms a perfect matching, then the edge set of the complete graph $K_{2k}$ can be partitioned into rainbow spanning trees. A strengthening was proposed by Kaneko, Kano, and Suzuki~\cite{kaneko2003three}, stating that for any proper edge-coloring of $K_n$ with $n\geq 5$, there exists $\lfloor n/2\rfloor$ disjoint rainbow spanning trees. Constantine~\cite{constantine2004edge} suggested that the spanning trees in the Brualdi-Hollingsworth conjecture can be chosen to be isomorphic to each other. Recently, Glock, K\"uhn, Montgomery, and Osthus~\cite{glock2021decompositions} verified the conjecture as well as its strengthening by Constantine for large enough values of $k$. 

The existence of disjoint rainbow spanning trees was also considered for not necessarily proper edge-colorings. Akbari and Alipour \cite{akbari2007multicolored} showed that each $K_n$ that is edge-colored such that no color appears more than $n/2$ times contains at least two disjoint rainbow spanning trees. Under the same assumption, Carraher, Hartke, and Horn~\cite{carraher2016edge} verified the existence of $\Omega(n/\log n)$ disjoint rainbow spanning trees. 

In contrast to the extensive list of results on complete graphs, not much is known for general graphs. In the past years, arborescence packing problems have seen renewed interest due to the fact that branchings form the common independent sets of a graphic matroid and a partition matroid, hence all positive results characterize algorithmically tractable instances of the problem of packing common bases. A milestone result of this area is Edmonds' arborescences theorem~\cite{edmonds1973edge}, characterizing the existence of $k$ pairwise disjoint arborescences in a directed graph. B\'erczi and Schwarcz~\cite{berczi2022rainbow} studied rainbow circuit-free colorings of binary matroids in general, and showed that if an $n$-element rank $r$ binary matroid $M$ is colored with exactly $r$ colors, then $M$ either contains a rainbow colored circuit or a monochromatic cocircuit. Such a coloring can be identified with a so-called reduction to a partition matroid, which is closely related to the problem of packing rainbow spanning trees. For further details on recent developements, see e.g.~\cite{horn2017many,fu2018number,horn2018rainbow,montgomery2019decompositions,pokrovskiy2018linearly}.

\paragraph{Our results.}

Despite impressive achievements, the complexity of finding disjoint rainbow spanning trees remained an intriguing open question that was raised by many, see e.g.~\cite{overflow,berczi2022rainbow}.

\searchprob{Packing Rainbow Spanning Trees}{Edge-colored graph $G=(V,E)$ and positive integer $k\in\mathbb{Z}_+$.}{Find $k$ pairwise disjoint rainbow spanning trees.}

Our first main result is a proof showing that {\sc Packing Rainbow Spanning Trees} is NP-complete even when $k=2$, the graph is the union of two spanning trees, and each color class contains exactly two edges\footnote{Just before the submission of the present paper, H\"orsch, Kaiser, and Kriesell~\cite{horsch2022rainbow} published an independent work that considers analogous problems.}. Therefore, determining $\beta(M_1,M_2)$ and $\gamma(M_1,M_2)$ is hard even when $M_1$ is a graphic matroid and $M_2$ is a partition matroid. 

As a corollary, we give a negative answer to a problem appeared in~\cite{egres}. A graph $G=(V,E)$ is called \emph{$k$-partition-connected} if for every partition $\mathcal{P}$ of $V$ there are at least $k(|\mathcal{P}|-1)$ edges in $E$ that connect different classes of $\mathcal{P}$. Tutte~\cite{tutte1961problem} showed that this is equivalent to the property that $G$ can be decomposed into $k$ connected spanning subgraphs. 

\searchprob{Decomposition of $k$-partition-connected Digraphs}{Digraph $D=(V,A)$ whose underlying graph is $k$-partition-connected in which all but one vertex have in-degree at least $k$, the remaining vertex $r\in V$ has in-degree $0$.}{Find $k$ weakly connected spanning subgraphs of $D$ so that every vertex $v\in V-r$ has positive in-degree in each of them.}

The problem can be rephrased using matroid terminology as follows. Let $M_1$ be the graphic matroid of the underlying undirected graph of $D$, and let $M_2$ be the partition matroid where each class consists of the set of edges entering a given vertex $v\in V-r$. Is it true that if both $M_1$ and $M_2$ contain $k$ disjoint spanning sets, then there are $k$ disjoint common spanning sets? When the degree of each vertex except $r$ is exactly $k$ then the digraph is in fact rooted $k$-edge-connected, and the existence of the subgraphs is question follows from Edmonds' disjoint arborescences theorem~\cite{edmonds1973edge}. Unfortunately, this seems to be the only tractable case, as we show that deciding the existence of $k$ subgraphs satisfying the conditions of the problem is NP-complete.

Given a matroid $M=(E,\mathcal{I})$ whose ground set is partitioned into two-element subsets called \emph{pairs}, a set $X\subseteq E$ is called a \emph{parity set} if it is the union of pairs. The \emph{matroid parity problem}, introduced by Lawler~\cite{lawler1976combinatorial}, asks for a parity independent set of maximum size. Though matroid parity cannot be solved efficiently in general matroids~\cite{jensen1982complexity,lovasz1981matroid}, Lovász~\cite{lovasz1981matroid} developed a polynomial time algorithm for linear matroids that is applicable if a linear representation is available. The graphic matroid of a graph is linear whose representation is easy to construct, therefore the following arises naturally.

\searchprob{Packing Parity Spanning Trees}{Graph $G=(V,E)$ whose edges are partitioned into pairs and positive integer $k\in\mathbb{Z}_+$.}{Find $k$ pairwise disjoint parity spanning trees.}

We prove hardness of {\sc Packing Parity Spanning Trees} by reduction from {\sc Packing Rainbow Spanning Trees}. Interestingly, the direction of the reduction between the two problems is just the opposite of the one appearing in \cite{berczi2021complexity} where the hardness of packing disjoint common bases was proved by reduction from packing parity bases. 

\paragraph{Paper organization.}

The rest of the paper is organized as follows. Basic definitions and notation are given in Section~\ref{sec:prelim}. The complexity of {\sc Packing Rainbow Spanning Trees} is discussed in Section~\ref{sec:rainbow}. Finally, we prove hardness of {\sc Decomposition of k-partition-connected digraphs} in Section~\ref{sec:kpart} and of {\sc Packing Parity Spanning Trees} in Section~\ref{sec:parity}.

\section{Preliminaries}
\label{sec:prelim}

\paragraph{General notation.} 
We denote the set of non-negative integers by $\mathbb{Z}_+$. For a positive integer $k$, we use $[k]\coloneqq \{1,\dots,k\}$. Given a ground set $E$ together with subsets $X,Y\subseteq E$, the \emph{difference} of $X$ and $Y$ is denoted by $X-Y$. If $Y$ consists of a single element $y$, then $X-\{y\}$ and $X\cup\{y\}$ are abbreviated as $X-y$ and $X+y$, respectively.

Let $G=(V,E)$ be a
directed graph.
The \emph{set of edges entering a vertex $v\in V$} is denoted by $\delta^{in}_G(v)$.
Note that 
$\delta^{in}(v)$ contains, if exist, the loops at $v$. 
The graph is \emph{weakly connected} if its underlying undirected graph is connected.

\paragraph{Matroids.}

Although the results are presented using graph terminology, we give a brief introduction into matroid theory; for further details, the interested reader is referred to \cite{oxley2011matroid}. Matroids were introduced independently by Whitney \cite{whitney1935abstract} and by Nakasawa \cite{nishimura2009lost}. A \emph{matroid} $M=(E,\cI)$ is defined by its \emph{ground set} $E$ and its \emph{family of independent sets} $\cI\subseteq 2^S$ that satisfies the \emph{independence axioms}: (I1) $\emptyset\in\cI$, (I2) $X\subseteq Y,\ Y\in\cI\Rightarrow X\in\cI$, and (I3) $X,Y\in\cI,\ |X|<|Y|\Rightarrow\exists e\in Y-X\ s.t.\ X+e\in\cI$. The maximal independent subsets of $E$ are called \emph{bases}. A set $Z\subseteq E$ is a \emph{spanning set} of $M$ if it contains a basis of the matroid. 

A \emph{partition matroid} is a matroid $N=(S,\cJ)$ such that $\cJ=\{X\subseteq S:\ |X\cap S_i|\leq 1\ \text{for $i=1,\dots,q$}\}$ for some partition $S=S_1\cup\dots\cup S_q$.\footnote{In general, the upper bounds might be different for the different partition classes. As all the partition matroids used in the paper have all-ones upper bounds, we make this restriction without explicitly mentioning it.} For a graph $G=(V,E)$, the \emph{graphic matroid} $M=(E,\cI)$ of $G$ is defined as $\cI=\{F\subseteq E:\ F\ \text{does not contain a cycle}\}$, that is, a subset $F\subseteq E$ is independent if it is a forest.

\paragraph{Rainbow spanning trees.}

Let $G=(V,E)$ be an edge-colored graph. A subset $F\subseteq E$ of edges is \emph{rainbow colored} if each color appears at most once in it. For short, we call a rainbow colored spanning tree of $G$ a \emph{rainbow spanning tree}. Let $M_1$ denote the graphic matroid of $G$ and $M_2$ be the partition matroid defined by the color classes of the coloring. Then the rainbow forests correspond to the common independent sets of $M_1$ and $M_2$. Therefore, the existence of a single rainbow spanning tree is characterized by Edmonds' matroid intersection theorem~\cite{edmonds1970submodular}, which gives the following: there exists a rainbow spanning tree in $G$ if and only if for any partition $\mathcal{P}$ of $V$, the edges going between different classes of $\mathcal{P}$ use at least
$|\mathcal{P}|-1$ colors.

\section{Packing rainbow spanning trees}
\label{sec:rainbow}

This section is devoted to the proof of the main result of the paper, the hardness of {\sc Packing Rainbow Spanning Trees}. The proof is by reduction from Monotone Not-All-Equal 3-Sat in which each variable appears exactly four times. In such a problem, we are given a CNF formula in which no negated variable appears, all clauses contain exactly three distinct variables, and each variable appears exactly four times, and the goal is to decide whether there is a truth assignment such that for each clause at least one literal evaluates to true and at least one to false, respectively. This problem is known to be NP-complete, see e.g.~\cite{porschen2014xsat,darmann2019simplified}.

\begin{thm}\label{thm:rainbow}
{\sc Packing Rainbow Spanning Trees} is NP-complete even when $k=2$, the graph is the union of two spanning trees, and each color class contains exactly two edges.
\end{thm}
\begin{proof}
Let $\Phi=(U,\mathcal{C})$ be an instance of Monotone Not-All-Equal 3-Sat where $U=\{x_1,\dots,x_n\}$ is the set of variables and $\mathcal{C}=\{C_1,\dots,C_m\}$ is the set of clauses, and each variable $x_i$ is contained in exactly four members of $\mathcal{C}$. We construct an instance of {\sc Packing Rainbow Spanning Trees} as follows.

For each variable $x_i$, we add a complete graph on vertices $\{u^i_p,v^i_p,w^i_p,z^i_p\}$ to $G$ for $p\in[4]$, that is, for each variable four complete graphs on four vertices are added. For each clause $C_j$, we add a triangle on vertices $\{c^j_1,c^j_2,c^j_3\}$. Assume that $C_j$ contains variables $x_{i_1}$, $x_{i_2}$ and $x_{i_3}$. If $x_{i_q}$ is the $\ell$th occurrence of the variable $x_{i_q}$ with respect to the ordering of the clauses, then we add an edge between $z^{i_q}_\ell$ and $c^j_q$ for $q\in[3]$. Finally, we add an extra vertex $r$ and connect it to $u^i_\ell$ with two parallel edges for $i\in[n]$, $\ell\in[4]$; see Figure~\ref{fig:rainbow} for an example. Let $G=(V,E)$ denote the graph thus arising. Note that the number of vertices is $|V|=16\cdot n+3\cdot m+1$, while the number of edges is $|E|=32\cdot n+6\cdot m$. It is not difficult to check that the edge set of $G$ can be decomposed into two spanning trees. 

\begin{figure}[t]
    \centering
    \includegraphics[width=\textwidth]{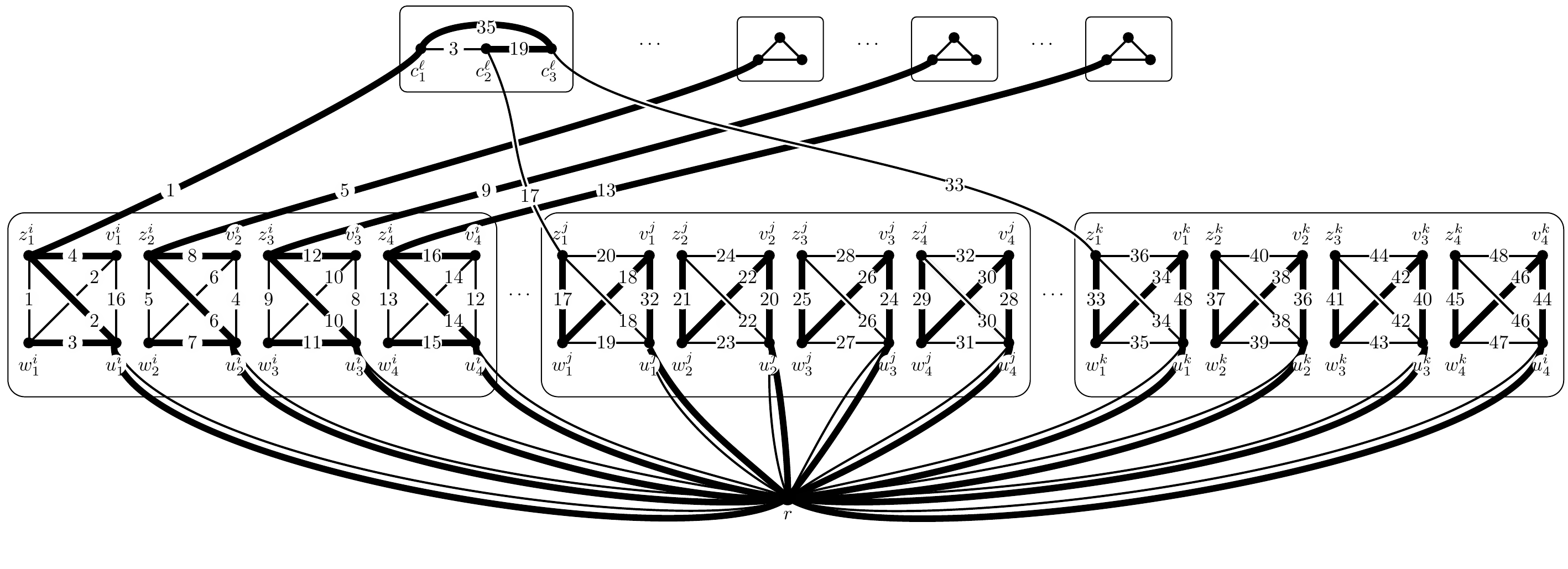}
    \caption{An illustration of the proof of Theorem~\ref{thm:rainbow}. Clause $C_\ell$ contains variables $x_i$, $x_j$ and $x_k$. Relevant color classes are represented by numbers. Thick and thin edges correspond to rainbow spanning trees $T_1$ and $T_2$, respectively, when the value of $x_i$ is set to true and the values of $x_j$ and $x_k$ are set to false.}
    \label{fig:rainbow}
\end{figure}

Now we define an edge-coloring of $G$ in which each color class consists of exactly two edges. The pairs of parallel edges leaving $r$ form distinct color classes. Consider a variable $x_i$, and let $C_{j_1}$, $C_{j_2}$, $C_{j_3}$ and $C_{j_4}$ be the clauses containing $x_i$. Furthermore, let $c^{j_p}_{q_p}$ be the neighbor of $z^i_p$ in the triangle corresponding to clause $C_{j_p}$ for $p\in[4]$. Then the coloring contains the pairs $\{w^i_pz^i_p,z^i_pc^{j_p}_{q_p}\}$, $\{v^i_pz^i_p,u^i_{p+1}v^i_{p+1}\}$, $\{u^i_{p}z^i_{p},v^i_{p}w^i_{p}\}$ and $\{u^i_pw^i_p,c^{j_p}_{q_p}c^{j_p}_{q_p+1}\}$ as color classes, where indices are meant in a cyclic order. It is worth mentioning that among these, only the first two pairs play a role in the reduction, and the remaining two are added only to ensure that each color class has size exactly two. 

We claim that $\Phi$ has a not-all-equal truth assignment if and only if the edge-set of $G$ can be partitioned into two rainbow spanning trees. To see this, we prove the two directions separately.

\begin{cl}\label{cl:1}
If $\Phi$ has a not-all-equal truth assignment, then the edge-set of $G$ can be partitioned into two rainbow spanning trees.
\end{cl}
\begin{proof}
Take an assignment of truth values to the variables such that each clause contains at least one true and one false variable. We construct two spanning trees $T_1$ and $T_2$ as follows. 

Obviously, for each pair of parallel edges leaving $r$, we add one of them to $T_1$. Let $x_i$ be a variable, $C_{j_1}$, $C_{j_2}$, $C_{j_3}$ and $C_{j_4}$ be the clauses containing $x_i$, and $c^{j_p}_{q_p}$ be the neighbor of $z^i_p$ in the triangle corresponding to clause $C_{j_p}$ for $p\in[4]$. If $x_i$ is a true variable, then we add the edges $u^i_pw^i_p$, $u^i_pz^i_p$, $v^i_pz^i_p$ and $z^i_pc^{j_p}_{q_p}$ to $T_1$ for $p\in[4]$. If $x_i$ is a false variable, then we add the edges $u^i_pv^i_p$, $v^i_pw^i_p$, $w^i_pz^i_p$, and $c^{j_p}_{q_p}c^{j_p}_{q_p+1}$ to $T_1$ for $p\in[4]$. Finally, define $T_2\coloneqq E-T_1$.

Observe that both $T_1$ and $T_2$ are rainbow colored by definition, hence it remains to show that both $T_1$ and $T_2$ are spanning trees. As each clause of $\Phi$ contains three variables and each variable is contained in exactly four clauses, we have $4\cdot n=3\cdot m$. Since $|T_i|=20\cdot n=16\cdot n+3\cdot m=|V|-1$, it suffices to show that $T_i$ forms a connected graph on $V$ for $i=1,2$, or equivalently, for any vertex $v\in V$, there exists a $v-r$ path in $T_i$. We show this for $T_1$; the proof for $T_2$ goes similarly.

The statement clearly holds for vertices of the $K_4$ subgraphs corresponding to the variables. Let $C_j$ be a clause and, for $q\in[3]$, let $x_{i_q}$ be the variable in $C_j$ such that $c^j_q$ has a neighbour in one of the complete graphs on four vertices corresponding to variable $x_{i_q}$, denoted by $z^{i_q}_{p_q}$. If $x_{i_q}$ is a true variable, then the edge $z^{i_q}_{p_q}c^j_q$ is in $T_1$, hence $c^j_q$ can reach $r$ in $T_1$. If $x_{i_q}$ is a false variable, then the edge $c^j_qc^j_{q+1}$ is in $T_1$. Now if $x_{i_{q+1}}$ is a true variable, then the edge $z^{i_{q+1}}_{p_{q+1}}c^j_{q+1}$ is in $T_1$, hence $c^j_q$ can reach $r$ in $T_1$ through $c^j_{q+1}$. Otherwise, the edge $c^j_{q+1}c^j_{q+2}$ is in $T_1$. As each clause contains at least one true variable, $x_{i_{q+2}}$ necessarily has true value, and so the edge $z^{i_{q+2}}_{p_{q+2}}c^j_{q+2}$ is in $T_1$, hence $c^j_q$ can reach $r$ in $T_1$ through $c^j_{q+1}$ and then $c^j_{q+2}$. This concludes the proof of the claim.
\end{proof}

\begin{cl}\label{cl:2}
If the edge-set of $G$ can be partitioned into two rainbow spanning trees, then $\Phi$ has a not-all-equal truth assignment.
\end{cl}
\begin{proof}
Take a partitioning of the edge-set of $G$ into two rainbow spanning trees $T_1$ and $T_2$. Let $x_i$ be a variable, $C_{j_1}$, $C_{j_2}$, $C_{j_3}$ and $C_{j_4}$ be the clauses containing $x_i$, and $c^{j_p}_{q_p}$ be the neighbor of $z^i_p$ in the triangle corresponding to clause $C_{j_p}$ for $p\in[4]$. 

We claim that the edges $z^i_pc^{j_p}_{q_p}$ are all contained in the same spanning tree $T_1$ or $T_2$; the truth assignment will be based on this distinction. Without loss of generality, let us assume that $z^i_1c^{j_1}_{q_1}$ is in $T_1$. By the definition of the color classes, this implies that $w^i_1z^i_1$ is in $T_2$. Since $u^i_1z^1_1$ and $v^i_1w^i_1$ have the same color, they are contained in different spanning trees. Any partitioning of a complete graph on four vertices into two spanning trees consists of two paths of length three, hence necessarily $u^i_1v^i_1$ is in $T_2$ and $v^i_1z^i_1$ is in $T_1$. The latter implies that $u^i_2v^i_2$ is in $T_2$ which, using similar arguments, shows that $z^i_2c^{j_2}_{q_2}$ and $v^i_2z^i_2$ are in $T_1$. Continuing this, we get that $u^i_3v^i_3$ is in $T_2$, $z^i_3c^{j_3}_{q_3}$ and $v^i_3z^i_3$ are in $T_1$, $u^i_4v^i_4$ is in $T_2$, $z^i_4c^{j_4}_{q_4}$ and $v^i_4z^i_4$ are in $T_1$. This proves that each of the edges $z^i_pc^{j_p}_{q_p}$ are all contained in the same spanning tree, namely in $T_1$ in this case.

We define a truth assignment as follows. If the edges $z^i_pc^{j_p}_{q_p}$ are contained in $T_1$, then we set the value of $x_i$ to true, otherwise we set it to false. Let $C_j$ be a clause and, for $q\in[3]$, let $x_{i_q}$ be the variable in $C_j$ such that $c^j_q$ has a neighbour in one of the complete graphs on four vertices corresponding to variable $x_{i_q}$, denoted by $z^{i_q}_{p_q}$. As both $T_1$ and $T_2$ are spanning trees, we have $T_i\cap \{c^j_1z^{i_1}_{q_1},c^j_2z^{i_2}_{q_2},c^j_3z^{i_3}_{q_3}\}\neq\emptyset$ for $i=1,2$, meaning that each clause contains at least one true and at least one false variable as required.
\end{proof}

The theorem follows by Claims~\ref{cl:1} and~\ref{cl:2}.
\end{proof}

\section{Further results}
\label{sec:cor}

As an application of Theorem~\ref{thm:rainbow}, in this section we prove hardness of {\sc Decomposition of $k$-partition-connected Digraphs} and {\sc Packing Parity Spanning Trees}.

\subsection{Decomposition of $k$-partition-connected digraphs}
\label{sec:kpart}

In matroid terms, {\sc Decomposition of $k$-partition-connected Digraphs} aims at finding pairwise disjoint common spanning sets of two matroids, one of them being a graphic matroid while the other is a partition matroid. When the in-degree of each vertex $v\in V-r$ is exactly $k$, then the existence of $k$ pairwise disjoint arborescences rooted at $r$ follows by Edmonds' disjoint arborescences theorem~\cite{edmonds1973edge}, and these arborescences are actually common spanning sets of the underlying two matroids. However, we show that deciding the existence of $k$ disjoint common spanning sets is hard in general.

\begin{thm}\label{thm:decomp}
{\sc Decomposition of $k$-partition-connected Digraphs} is NP-complete even when $k=2$.
\end{thm}
\begin{proof}
Consider an instance of {\sc Packing Rainbow Spanning Trees} where the graph $G=(V,E)$ is the union of two spanning trees, and each color class contains exactly two edges. For ease of discussion, we denote the set of colors by $\mathcal{C}$. We construct an instance $D=(U,F)$ of {\sc Decomposition of $2$-partition-connected Digraphs} as follows.

Let $r\in V$ be an arbitrary vertex of $G$. As $E$ can be decomposed into two spanning trees, $G$ admits an orientation in which each vertex has in-degree exactly two, except $r$ having in-degree $0$. Let $\overrightarrow{G}=(V,\overrightarrow{E})$ denote the digraph obtained by taking such an orientation. For each vertex $v\in V-r$, we add three vertices $v^{in}_1,v^{in}_2,v^{out}$ to $U$. Furthermore, we add a copy of $r$ to $U$ that, by abuse of notation, we denote also by $r$. For $i=1,2$, we add two parallel arcs from $v^{in}_i$ to $v^{out}$ for each vertex $v\in V-r$. Furthermore, if $u_1v$ and $u_2v$ are the two arcs entering $v$ in $\overrightarrow{G}$, then we add the arcs $u^{out}_1v^{in}_1$ and $u^{out}_2v^{in}_2$ to $F$ if $u_1,u_2\neq r$, while if one of them (or even both), say $u_1$, is $r$ then we add the arc $rv^{in}_1$ instead. 
This way, each edge of $G$ has a corresponding 'image' in $D$.
Finally, if the edges $uv,xy\in E$ formed a color class $c\in\mathcal{C}$ in the original rainbow spanning tree instance and the corresponding arcs in $D$ are $u^{out}v^{in}_p$ and $x^{out}y^{in}_q$ for some $p,q\in[2]$, then we add a vertex $w_c$ to $U$ together with arcs $w_cv^{in}_p,w_cy^{in}_q$ and two loops on $w_c$; see Figure~\ref{fig:decomp} for an example. 

\begin{figure}[t!]
\centering
\begin{subfigure}[t]{0.48\textwidth}
  \centering
  \includegraphics[width=.8\linewidth]{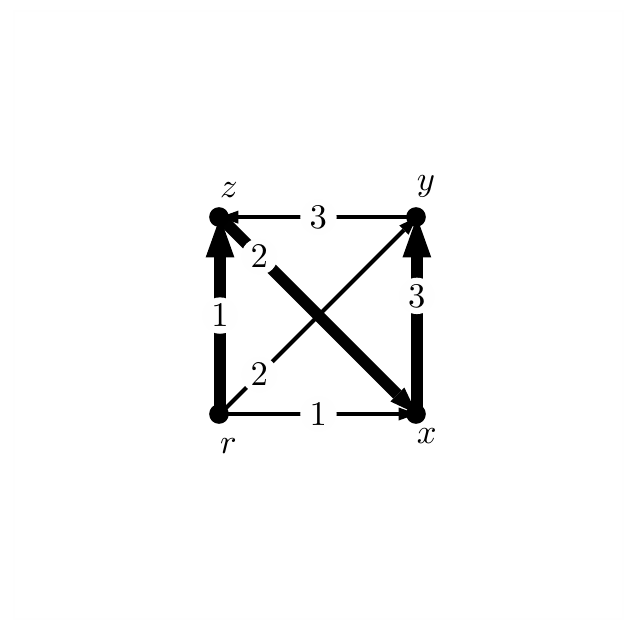}
  \caption{An instance of {\sc Packing Rainbow Spanning Trees}. Color classes are represented by numbers. Thick and thin edges form rainbow spanning trees $T_1$ and $T_2$, respectively. The edges are oriented in such a way that the in-degree of each vertex in $V-r$ is two.}
  \label{fig:dec1}
\end{subfigure}\hfill
\begin{subfigure}[t]{0.48\textwidth}
  \centering
  \includegraphics[width=.8\linewidth]{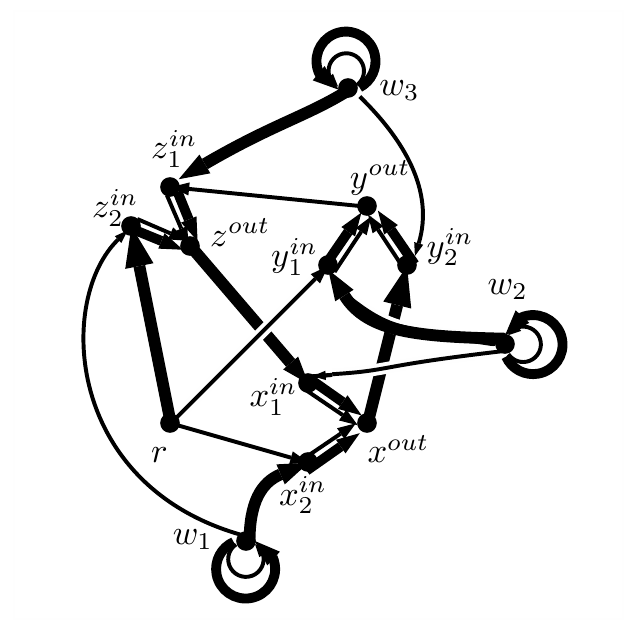}
  \caption{The corresponding {\sc Decomposition of $k$-partition-connected Digraphs} instance $D=(U,A)$. Thick and thin edges form connected subgraphs $F_1$ and $F_2$ with each vertex in $U-r$ having positive in-degree.}
  \label{fig:dec2}
\end{subfigure}
\caption{An illustration of Theorem~\ref{thm:decomp}.}
\label{fig:decomp}
\end{figure}

We claim that the edge-set of $G$ can be decomposed into two rainbow spanning trees if and only if $A$ can be partitioned into two weakly connected spanning subgraphs such that every vertex $u\in U-r$ has positive in-degree in both of them. To see this, we prove the two directions separately.

\begin{cl}\label{cl:3}
If the edge-set of $G$ can be partitioned into two rainbow spanning trees, then $A$ can be partitioned into two weakly connected spanning subgraphs such that every vertex $u\in U-r$ has positive in-degree in both of them.
\end{cl}
\begin{proof}
Take a partitioning of the edge-set of $G$ into two rainbow spanning trees $T_1$ and $T_2$. For $i=1,2$, define $F_i$ as follows. Take the arcs of $D$ corresponding to the edges of $T_i$. Add one arc from each parallel pair going between $v^{in}_i$ and $v^{out}$ for $i=1,2$, $v\in V-r$. For each color class $c\in\mathcal{C}$, add one of the loops on $w_c$ to $F_i$. Finally, since exactly one edge of $c$ appeared in $T_i$, exactly one of the two neighbours of $w_c$ has in-degree zero in $F_i$ so far;  
add the arc leaving $w_c$ that goes to this vertex. It is easy to check that $F_i$ is a weakly connected 
and each vertex $U-r$ has positive in-degree in it, concluding the proof of the claim. 
\end{proof}

\begin{cl}\label{cl:4}
If $A$ can be partitioned into two weakly connected spanning subgraphs such that every vertex $u\in U-r$ has positive in-degree in both of them, then the edge-set of $G$ can be partitioned into two rainbow spanning trees.
\end{cl}
\begin{proof}
Take a partitioning of $A$ into two weakly connected spanning subgraphs $F_1$ and $F_2$ such that every vertex $u\in U-r$ has positive in-degree in both of them. For each color class $c\in \mathcal{C}$, the vertex $w_c$ is connected to the rest of $D$ through two outgoing arcs, hence one of them is contained in $F_1$ while the other is contained in $F_2$. This implies two things: the set $T_i\subseteq E$ of original edges whose images are contained in $F_i$ form a connected subgraph of $G$ for $i=1,2$, and the images of the two edges in color class $c$ are contained in different $F_i$s due to the in-degree constraints, hence the same holds for their original counterpart in $E$ with the $T_i$s. As $|E|=2\cdot |V|-2$, $T_1$ and $T_2$ necessarily partitions the edge-set of $G$ into rainbow spanning trees as required.
\end{proof}

The theorem follows by Theorem~\ref{thm:rainbow} and Claims~\ref{cl:3} and~\ref{cl:4}.
\end{proof}

\subsection{Packing parity spanning trees}
\label{sec:parity}

In the parity base packing problem we are given a matroid $M=(E,\mathcal{I})$ whose ground-set is partitioned into two-element subsets, and the goal is to find pairwise disjoint parity bases of $M$. The problem was shown to be hard in~\cite{berczi2021complexity}, which in turn implied the hardness of packing disjoint common bases of two matroids. However, the proof there did not settle the complexity of packing parity bases in graphic matroids, which remained an interesting open problem. 

Now we close this gap by showing that the problem remains hard in graphic matroids. The reduction is from {\sc Packing Rainbow Spanning Trees}; note that the two problems are orthogonal in the sense that the members of a pair should be contained in the same spanning tree in {\sc Packing Parity Spanning Trees}, while in different spanning trees in {\sc Packing Rainbow Spanning Trees}. 

\begin{thm}\label{thm:parity}
{\sc Packing Parity Spanning Trees} is NP-complete even when $k=2$ and the graph is the union of two spanning trees.
\end{thm}
\begin{proof}
Consider an instance of {\sc Packing Rainbow Spanning Trees} where the graph $G=(V,E)$ is the union of two spanning trees, and each color class contains exactly two edges. For ease of discussion, we denote the set of colors by $\mathcal{C}$. We construct an instance $G'=(V',E')$ of {\sc Packing Parity Spanning Trees} as follows.

We extend $G$ by adding a new vertex $w_c$ for each color class $c\in\mathcal{C}$. Furthermore, if the color class $c$ contains the edges $e,f$, then we add an edge $e_c$ between $w_c$ and one of the end vertices of $e$, and an edge $f_c$ between $w_c$ and one of the end vertices of $f$. Finally, we define the pairs of $e$ and $f$ to be $e_c$ and $f_c$, respectively; see Figure~\ref{fig:parity} for an example. Let $G'=(V',E')$ denote the graph thus obtained. Note that $V\subseteq V'$ and $E\subseteq E'$. 

\begin{figure}[t!]
\centering
\begin{subfigure}[t]{0.48\textwidth}
  \centering
  \includegraphics[width=.55\linewidth]{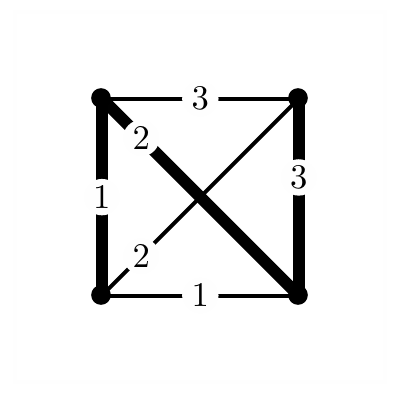}
  \caption{An instance of {\sc Packing Rainbow Spanning Trees}. Color classes are represented by numbers. Thick and thin edges form rainbow spanning trees $T_1$ and $T_2$, respectively.}
  \label{fig:par1}
\end{subfigure}\hfill
\begin{subfigure}[t]{0.48\textwidth}
  \centering
  \includegraphics[width=.55\linewidth]{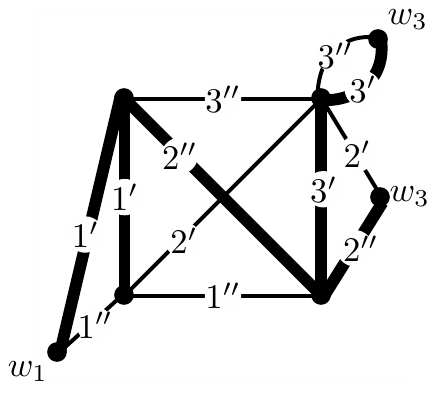}
  \caption{The corresponding {\sc Packing Parity Spanning Trees} instance $G'=(V',E')$. Edges in the same pair share the same number. Thick and thin edges form parity spanning trees $T'_1$ and $T'_2$.}
  \label{fig:par2}
\end{subfigure}
\caption{An illustration of Theorem~\ref{thm:parity}.}
\label{fig:parity}
\end{figure}

Note that since $G$ is the union of two spanning trees, the same holds for $G'$. Furthermore, it is obvious to check that $T'_1$ and $T'_2$ give a partitioning of the edge-set of $G'$ into two parity spanning trees if and only if $T_1\coloneqq T'_1\cap E$ and $T_2\coloneqq T'_2\cap E$ give a partitioning of the edge-set of $E$ into two rainbow spanning trees, hence the theorem follows.
\end{proof}

\section*{Acknowledgement} The work was supported by the Lend\"ulet Programme of the Hungarian Academy of Sciences -- grant number LP2021-1/2021 and by the Hungarian National Research, Development and Innovation Office -- NKFIH, grant numbers FK128673 and TKP2020-NKA-06.

\bibliographystyle{abbrv}
\bibliography{packing_rainbow_trees}

\end{document}